\documentclass[11pt,letterpapert]{article}
\usepackage{graphicx}
\usepackage{amssymb}
\usepackage{epstopdf}
\usepackage{fancyhdr}
\usepackage{amsmath}
\usepackage{amsthm}
\usepackage{hyperref}
\usepackage{makeidx}
\usepackage{blkarray}
\usepackage{mathdesign}
\usepackage{color}
\usepackage{geometry}
\usepackage{soul}
\usepackage{nameref}
\usepackage{ulem}
\usepackage{mathtools}

\usepackage{tikz}

\definecolor{purple}{rgb}{.9,0,.9}

\definecolor{newred}{RGB}{180,20,5}

\definecolor{newgreen}{RGB}{1,129,30}

\graphicspath{{./figures/}}

\makeatletter
\let\orgdescriptionlabel\descriptionlabel
\renewcommand*{\descriptionlabel}[1]{%
  \let\orglabel\label
  \let\label\@gobble
  \phantomsection
  \edef\@currentlabel{#1}%
  \let\label\orglabel
  \orgdescriptionlabel{#1}%
}
\makeatother

\graphicspath{{./figures/}}

\newcommand{\Nat}{\mathbb{N}}

\newcommand{\Real}{\mathbb{R}}

\newcommand{\cB}{{\cal B}}\newcommand{\cC}{{\cal C}}

\newcommand{\cH}{{\cal H}}
\newcommand{\cL}{{\cal L}}
\newcommand{\cM}{{\cal M}}

\newcommand{\cU}{{\cal U}}
\newcommand{\cW}{{\cal W}}

\newcommand{\ve}{\varepsilon}

\newtheorem{theorem}{Theorem}[section]
\newtheorem{lemma}[theorem]{Lemma}

\newtheorem{proposition}[theorem]{Proposition}

\newtheorem{remark}[theorem]{Remark}

\def\Xint#1{\mathchoice
{\XXint\displaystyle\textstyle{#1}}%
{\XXint\textstyle\scriptstyle{#1}}%
{\XXint\scriptstyle\scriptscriptstyle{#1}}%
{\XXint\scriptscriptstyle\scriptscriptstyle{#1}}%
\!\int}
\def\XXint#1#2#3{{\setbox0=\hbox{$#1{#2#3}{\int}$ }
\vcenter{\hbox{$#2#3$ }}\kern-.6\wd0}}

\def\dashint{\Xint-}

\newcommand{\beqn}{\begin{equation}}
\newcommand{\eeqn}{\end{equation}}

\newcommand{\norm}[1]{\left\|#1\right \|}

\def\weakstar{\stackrel{*}{\rightharpoonup}}
\newcommand{\mres}{\mathbin{\vrule height 1.6ex depth 0pt width
		0.13ex\vrule height 0.13ex depth 0pt width 1.3ex}}  

\title{A beam that can only bend on the Cantor set}

\author{
Roberto Paroni$^1$ \!\!\!\!\! \and Brian Seguin*$^2$
}

\begin{document}

\maketitle

\vspace{-1cm}
\begin{center}
{\small
$^1$ DICI,
Universit\`a di Pisa\\ 
Largo Lucio Lazzarino 1, 56122 Pisa, Italy\\
\href{mailto:roberto.paroni@unipi.it}{roberto.paroni@unipi.it}\\[8pt]
$^2$  Department of Mathematics and Statistics\\ 
Loyola University Chicago, Chicago, IL 60660, USA\\
\href{mailto:bseguin@luc.edu}{bseguin@luc.edu}\\[8pt]
}
\end{center}


\begin{abstract}
In this work we address the following question: is it possible for a one-dimensional, linearly elastic beam to only bend on the Cantor set and, if so, what would the bending energy of such a beam look like? We answer this question by considering a sequence of beams, indexed by $n$, each one only able to bend on the set associated with the $n$-th step in the construction of the Cantor set and compute the $\Gamma$-limit of the bending energies. 
The resulting energy in the limit has a structure similar to the traditional bending energy, a key difference being that the measure used for the integration is the Hausdorff measure of dimension $\ln 2/\ln 3$, which is the dimension of the Cantor set.
\end{abstract}

\section{Introduction}
Our main goal is to understand the deflection of a one-dimensional, linearly elastic beam that can only bend on a Cantor set subject to a particular choice of boundary conditions and loads.  To achieve this, we consider a sequence of beams, parametrized by $n\in\mathbb{N}$, that can only bend on the set corresponding to the $n$-th step in the construction of the Cantor set. See Figure \ref{fig1}.

\begin{figure}[h!]
\centering
\begin{tikzpicture}[scale=0.9]
\def \s{-0.35}
\draw[thick,blue] (0,0) -- (6,0) node[right, black] {n=0}; 
\draw[thick,blue] (0,\s) -- (6,\s)  node[right, black] {n=1}; 
\draw[ultra thick,black] (2,\s) -- (4,\s);
\draw[thick,blue] (0,2*\s) -- (6,2*\s) node[right, black] {n=2}; 
\draw[ultra thick,black] (2,2*\s) -- (4,2*\s);
\draw[ultra thick,black] (0.67,2*\s) -- (2-0.67,2*\s);
\draw[ultra thick,black] (4.67,2*\s) -- (6-0.67,2*\s);
\draw[thick,blue] (0,3*\s) -- (6,3*\s)  node[right, black] {n=3}; 
\draw[ultra thick,black] (2,3*\s) -- (4,3*\s);
\draw[ultra thick,black] (0.67,3*\s) -- (2-0.67,3*\s);
\draw[ultra thick,black] (4.67,3*\s) -- (6-0.67,3*\s);
\draw[ultra thick,black] (0.22,3*\s) -- (0.45,3*\s);
\draw[ultra thick,black] (2-0.45,3*\s) -- (2-0.22,3*\s);
\draw[ultra thick,black] (4.22,3*\s) -- (4.45,3*\s);
\draw[ultra thick,black] (6-0.45,3*\s) -- (6-0.22,3*\s);

\begin{scope}[xshift=230]

\draw[thick,blue] (0,0) -- (6,0) node[right, black] {n=2}; 
\draw[ultra thick,black] (2,0) -- (4,0);
\draw[ultra thick,black] (0.67,0) -- (2-0.67,0);
\draw[ultra thick,black] (4.67,0) -- (6-0.67,0);
\def \h{-10}
\begin{scope}[yshift=\h]
     \draw[domain=0:0.67, smooth, variable=\x, thick, blue] plot ({\x}, {(\x-18)*\x*\x/400});
     \draw[domain=2-0.67:2, smooth, variable=\x, thick, blue] plot ({\x}, {(\x-18)*\x*\x/400});
     \draw[ultra thick, black] (0.67, {(0.67-18)*0.67*0.67/400}) -- (1.33, {(1.33-18)*1.33*1.33/400});
     \draw[domain=4:4+0.67, smooth, variable=\x, thick, blue] plot ({\x}, {(\x-18)*\x*\x/400});
     \draw[ultra thick, black] (2, {(2-18)*2*2/400}) -- (4, {(4-18)*4*4/400});
      \draw[domain=6-0.67:6, smooth, variable=\x, thick, blue] plot ({\x}, {(\x-18)*\x*\x/400});
     \draw[ultra thick, black] (4.67, {(4.67-18)*4.67*4.67/400}) -- (5.33, {(5.33-18)*5.33*5.33/400});
\end{scope}   
\end{scope}
  \end{tikzpicture}
\caption{On the left the first three steps in the construction of the Cantor set are shown. The thin (blue) parts, that lead to the Cantor set, are deformable while the thick (black) parts are considered rigid. On the right it is depicted a possible deformation for a beam corresponding at the second step in the construction of the Cantor set.}\label{fig1}
\end{figure}
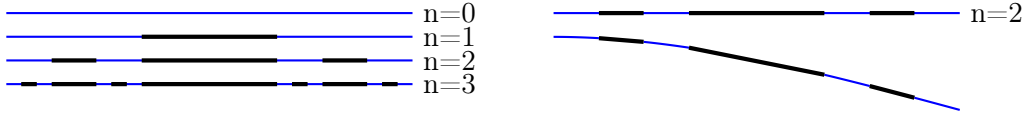
Let $\cC=\bigcap_{n\in\Nat} \cC_n$ be the Cantor set constructed from the interval $[0,\ell]$, where $\cC_n$ is the set corresponding to the $n$-th step in the construction of the Cantor set. To the $n$-beam, i.e., the beam corresponding to the $n$-th step, we associate a bending elastic energy $E^{\tiny el}_n$
defined by
\beqn
E^{\tiny el}_n(u)\coloneqq\tfrac{1}{2}\int_{0}^\ell b_n |u''|^2 d\cL,
\eeqn
where $b_n$ is the bending stiffness, $u$ is the transversal displacement, and $\ell$ is the length of the beam.
The part of the beam off of $\cC_n$ is modeled as rigid by setting the bending stiffness equal to $+\infty$ off of $\cC_n$. Since the measure of  $\cC_n$  goes to zero, if we keep the bending stiffness on $\cC_n$ independent of $n$, the beam would become more and more rigid as $n$ increases---that is, if we apply the same load to each $n$-beam the sequence of the maximum displacements will tend to zero. Said differently, the limiting beam would be rigid. Therefore, to avoid what would be a trivial result, we appropriately let $b_n$ on $\cC_n$ decrease to zero. We show that  by setting  $b_n=(2/3)^n b$, with $b$ a constant, the sequence of energies is bounded above and the limiting beam is not rigid. In passing, we note that the measure of $\cC_n$ is exactly $(2/3)^n \ell$, which partly justifies the scaling of the bending stiffness.

With these choices made, we use the theory of $\Gamma$-convergence to determine the energy of the limiting beam. We refer to Section \ref{sec3} for the specification of the topology under which the $\Gamma$-limit is taken, we here only describe the results informally. The limiting energy turns out to be finite only for displacements $u$ that are continuous, differentiable, and whose second derivative is a Radon measure.  Also, $u''$ is absolutely continuous with respect to the $s$-dimensional Hausdorff measure restricted to the Cantor set $\cC$, where $s=\ln 2 / \ln 3$ is the Hausdorff dimension of the Cantor set. In mechanical terms we can say that the limiting beam can ``bend'' only on the Cantor set, that has measure zero with respect to the Lebesgue measure. Setting $k_\cC$ equal to the Radon-Nikod\'ym derivative of $u''$ with respect to the $s$-dimensional Hausdorff measure restricted to the Cantor set $\cC$, the limiting elastic energy that we find is
\beqn
E^{el}_\infty(u)= \tfrac{1}{2}\int_\cC \frac{b\cH^s(\cC)}{\ell} \, k_\cC^2 \, d\cH^s,
\eeqn
where $\cH^s$ is the $s$-dimensional Hausdorff measure. From the limiting energy, but also from the definition, we deduce that  $k_\cC$ is the (linearized) curvature of the beam. This curvature is ``concentrated'' on the Cantor set.

\section{Preliminaries}
We here  introduce the notation and some of the results that will be used in the sections that follow. The proofs of all quoted results can be found in the book by Ambrosio, Fusco, and Pallara \cite{Ambrosio2000}.

We denote by $\cL$ and $\cH^s$ the Lebesgue and the $s$-dimensional Hausdorff measures on $\Real$, respectively. For any (real) measure $\mu$, we denote by $|\mu|$ the total variation of $\mu$ and by $\mu\mres \Omega$ the restriction of $\mu$ to the measurable set $\Omega$, i.e., $\mu\mres \Omega(A)=\mu(\Omega\cap A)$.  The push-forward of a measure $\mu$ by a function $\phi$ 
is denoted by $\phi_{\#}\mu$, i.e.,  $\phi_{\#}\mu(A)= \mu(\phi^{-1}(A))$. If a (real) measure $\mu$ is absolutely continuous with respect to a positive measure $\eta$, we write $\mu\ll\eta$ and denote the Radon-Nikod\'ym derivative by  $\frac{\mu}{\eta}$. Note that if
$
\gamma=\frac{\mu}{\eta},
$ than $
|\gamma|=\frac{|\mu|}{\eta}.
$
For an open set $\Omega\subseteq\Real$, we denote by $\mathcal{M}(\Omega)$ the space of Radon measures on $\Omega$. We say that a sequence $(\mu_n)\in \mathcal{M}(\Omega)$ weakly$^*$ converges to  $\mu\in \mathcal{M}(\Omega)$, and write
$\mu_n\weakstar \mu$ in $\mathcal{M}(\Omega)$, if 
$$
\lim_{n\to\infty}\int_\Omega \psi\,d\mu_n=\int_\Omega \psi\,d\mu
$$
for all continuous functions $\psi$ with compact support in $\Omega$, i.e., for all $\psi\in C_c(\Omega)$. 
If $\mu_n\weakstar \mu$ and $\phi$ is continuous, than $\phi_{\#}\mu_n\weakstar \phi_{\#}\mu,$. Also,
if $\mu_n\weakstar \mu$ and $|\mu_n|\weakstar \lambda$, than $\lambda\ge |\mu|$.

Given a bounded open subset $\Omega$ of $\Real$ and $1 \le m,p \le +\infty$, we use standard notation for the Sobolev and Lebesgue spaces $W^{m,p}(\Omega)$ and $L^p(\Omega)$. The space of functions of bounded variation on $\Omega$ is denoted by 
$$
BV(\Omega)=\{u\in L^1(\Omega)\ |\  u'\in \mathcal{M}(\Omega)\};
$$
we refer to \cite{Ambrosio2000} for the main properties of this space.

\section{Main result}\label{sec3}

The goal of this section is to understand the deflection of a one-dimensional, linearly elastic beam that can only bend on a Cantor set subject to a particular choice of boundary conditions and loads.  To achieve this, we consider a sequence of energies $E_n$ that model a beam that can only bend on the $n$-th step in the construction of the Cantor set and then compute the $\Gamma$-limit.  

Towards that end, consider a one-dimensional beam that has one end embedded in a rigid wall.  Assume the length of the beam outside of the wall is $\ell$ and the length inside is $\delta$.  Choose a coordinate system so that the material points of the beam can be described by $x\in(-\delta,\ell)$.  The displacement of the beam is described by a function $u:(-\delta,\ell)\rightarrow\Real$, where positive $u$ corresponds to a downward deflection.  A single point load in the downward direction of magnitude $P>0$ is applied to the beam at $x=\ell$.

The bending stiffness of the beam depends on $n$.  To describe this stiffness, we first introduce the $n$-th step in the construction of the middle third Cantor set on $[0,\ell]$.  To do so we introduce the functions $\psi_1,\psi_2:\Real\rightarrow\Real$ defined by
\beqn
\psi_1(x)\coloneqq\frac{x}{3}\quad\text{and}\quad \psi_2(x)\coloneqq\frac{x+2\ell}{3},\qquad x\in\Real.
\eeqn
Set $\cC_0\coloneqq[0,\ell]$, and iteratively define
\beqn
\cC_n\coloneqq\psi_1(\cC_{n-1})\cup\psi_2(\cC_{n-1}),\qquad n\in\Nat.
\eeqn
Notice that $\cC_n$ is the union of $2^n$ intervals each of length $3^{-n}\ell$.  Denote these intervals by $I_{n,i}$, $i\in\{1,\dots,2^n\}$, so that
\beqn\label{Cnint}
\cC_n=\bigcup_{n=1}^{2^n}I_{n,i}\qquad \text{and}\qquad |I_{n,i}|=3^{-n}\ell.
\eeqn
The Cantor set $\cC$ is then given by
\beqn
\cC=\bigcap_{n\in\Nat} \cC_n.
\eeqn
With this notation in place, we can specify the material properties of the beam at step $n$.  For a fixed $b>0$, consider the bending stiffness $b_n:(-\delta,\ell)\rightarrow\Real\cup\{\infty\}$ defined by
\beqn
b_n(x)\coloneqq\begin{cases}
\big(\tfrac{2}{3}\big)^n b & \text{if}\  x\in\cC_n,\\
\infty & \text{if}\  x\in(-\delta,\ell)\setminus \cC_n,
\end{cases}
\eeqn
and the energy functional $E_n:W^{2,2}(-\delta,\ell)\to \Real\cup\{+\infty\}$ by
\beqn\label{Endef}
E_n(u)\coloneqq
\begin{cases}
\displaystyle\tfrac{1}{2}\int_{0}^\ell b_n |u''|^2 d\cL-Pu(\ell) & \mbox{if } u\in \cW,\\
+\infty & \mbox{otherwise},
\end{cases}
\eeqn
where
\beqn
\cW:=\{u\in W^{2,2}(-\delta,\ell)\ |\ u=0\ \text{on}\ (-\delta,0]\}.
\eeqn
\begin{remark}
{\rm We have here chosen the simplest boundary and loading conditions. These can be easily changed. For instance, our main result still holds for any loading that can be taken into account with a functional that is continuous with respect to the convergence under which we  take the $\Gamma$-limit.
Also, we chose to implement the boundary condition by imposing the displacement $u=0$ on $(-\delta,0]$ in place of $u(0)=u'(0)=0$ because, as it will be soon clear,  in the point $x=0$ the function $u'$ may be discontinuous. }
\end{remark}

Having infinite bending stiffness off of $\cC_n$ suggests that the beam should not be able to bend on that set.  Our first result formulates and proves this fact.
\begin{proposition}\label{udpzero}
Let $n$ be fixed. If $u\in\cW$ and $E_n(u)<\infty$, then $u''=0$ a.e.~on $(-\delta,\ell)\setminus \cC_n$.
\end{proposition}

\begin{proof}
It follows from the definition of $\cW$ that $u(0)=u'(0)=0$.  Thus, we have
\beqn\label{basicuineq}
|u(\ell)|\leq \int_0^\ell |u'(y)|dy\leq \int_0^\ell \int_0^y|u''(x)|dxdy\leq \ell \int_{0}^\ell |u''(x)|dx,
\eeqn
and from Jensen's inequality,
\beqn
|u(\ell)|^2\leq \ell^3 \int_{0}^\ell |u''|^2d\cL.
\eeqn
So, for any $\alpha>0$, we find that
\begin{align}
\nonumber\infty>E_n(u)&\geq\tfrac{1}{2}\int_0^\ell b_n|u''|^2d\cL-\frac{\alpha}{2}|u(\ell)|^2-\frac{P^2}{2\alpha}\\
\label{udpzeromaineq}&\geq \tfrac{1}{2}\int_0^\ell \big (b_n-\alpha\ell^3\big) |u''|^2 d\cL-\frac{P^2}{2\alpha}.
\end{align}
Using $\alpha=\tfrac{1}{2}\big(\frac{2}{3}\big)^n\ell^{-3}b$ in the previous inequality allows us to conclude that
\beqn
\infty>\tfrac{1}{4}\int_0^\ell b_n|u''|^2d\cL.
\eeqn
Since $b_n$ is infinite off of $\cC_n$, this is only possible if $u''=0$ a.e.~on $(0,\ell)\setminus\cC_n$.  We know $u''=0$ on $(-\delta,0]$ from the fact that $u\in\cW$.
\end{proof}

Next we establish two preliminary lemmas that will be useful in obtaining a compactness result and the $\Gamma$-convergence.  The first one is a consequence of Theorem~2.34 of Ambrosio, Fusco, and Pallara \cite{Ambrosio2000}, but we include a simple proof here that works in the case we require.
\begin{lemma}\label{abscontm}
Let $(\mu_n),(\eta_n)\in \cM(-\delta,\ell)$ be sequences of Radon measures and  $\mu,\eta\in \cM(-\delta,\ell)$, with each $\eta_n$ and $\eta$ being positive, such that $\mu_n\weakstar \mu$, $\eta_n\weakstar \eta$, and $\mu_n\ll\eta_n$ for all $n\in\Nat$.  If
\beqn\label{supcond}
\sup_{n\in\Nat} \int_{-\delta}^\ell \Big(\frac{\mu_n}{\eta_n}\Big)^2 d\eta_n<\infty,
\eeqn
than $\mu\ll\eta$ and
\beqn\label{lemliminf}
\liminf_{n\rightarrow \infty} \int_{-\delta}^\ell \Big(\frac{\mu_n}{\eta_n}\Big)^2 d\eta_n\geq \int_{-\delta}^\ell \Big(\frac{\mu}{\eta}\Big)^2 d\eta.
\eeqn
\end{lemma}

\begin{proof}
To establish $\mu\ll\eta$, begin by fixing $\varphi\in C_c(-\delta,\ell)$ and set $\gamma_n\coloneqq\frac{\mu_n}{\eta_n}$.  Notice that
\begin{align}
\Big|\int_{-\delta}^\ell \varphi d|\mu_n|\Big|&\leq \int_{-\delta}^\ell |\varphi| d|\mu_n|= \int_{-\delta}^\ell |\varphi| |\gamma_n| d \eta_n\leq \Big(\int_{-\delta}^\ell \varphi^2d\eta_n\Big)^{1/2}\Big(\int_{-\delta}^\ell |\gamma_n|^2d\eta_n\Big)^{1/2}.
\end{align}
Utilizing \eqref{supcond}, we can take the limit $n\rightarrow \infty$ of the previous inequality to find that
\beqn\label{muetaineq}
\Big|\int_{-\delta}^\ell \varphi d|\mu|\Big|\leq K\Big(\int_{-\delta}^\ell \varphi^2d\eta\Big)^{1/2}
\eeqn
for some $K>0$.  Suppose that $\eta(\cB)=0$ for some Borel set $\cB$.  Fix $\ve>0$ and find $\cU$ open containing $\cB$ such that $\eta(\cU)<\ve$.  Now if $|\varphi|\leq 1$ and $\text{supp}\, \varphi\subseteq \cU$, we can use \eqref{muetaineq} to obtain
\beqn
\Big|\int_{-\delta}^\ell \varphi d|\mu|\Big|\leq K \ve^{1/2}.
\eeqn
Taking the supremum over all such $\varphi$ yields $|\mu|(\cB)\leq |\mu|(\cU)\leq K \ve^{1/2}$.  Since $\ve>0$ and $\cB$ were arbitrary, this shows that $\mu\ll\eta$.  We can now set $\gamma\coloneqq\frac{\mu}{\eta}$.

To establish \eqref{lemliminf}, start with $\varphi\in C_c(-\delta,\ell)$ and notice
\beqn
0\leq (\gamma_n-\varphi)^2=\gamma_n^2-2\gamma_n\varphi+\varphi^2.
\eeqn
Utilizing this, we find that
\begin{align*}
\liminf_{n\rightarrow\infty} \int_{-\delta}^\ell \gamma_n^2 d\eta_n&\geq \liminf_{n\rightarrow \infty} \int_{-\delta}^\ell (2\gamma_n\varphi-\varphi^2) d\eta_n\\
&=\liminf_{n\rightarrow \infty} \Big[ \int_{-\delta}^\ell 2\varphi d\mu_n-\int_{-\delta}^\ell\varphi^2 d\eta_n\Big]\\
&=\int_{-\delta}^\ell 2\varphi d\mu-\int_{-\delta}^\ell\varphi^2 d\eta\\
&=\int_{-\delta}^\ell [\gamma^2-(\gamma-\varphi)^2]d\eta.
\end{align*}
Since $C_c$ functions are dense in $L^2$ and $\varphi$ was arbitrary, this implies \eqref{lemliminf}.
\end{proof}

The next lemma characterizes the weak* limit of a sequence of measures associated with the construction of the Cantor set.
\begin{lemma}\label{etaconv}
The sequence of measures $\eta_n\coloneqq\big(\tfrac{3}{2}\big)^n \cL\mres \cC_n$, $n\in\Nat$, converges weak* to
\beqn\label{measrep}
\eta=\frac{\ell}{\cH^s(\cC)}\cH^s\mres\cC\quad\text{and}\quad \eta(I_{n,i})=2^{-n}\ell,
\eeqn
where $s=\ln 2 / \ln 3$.
\end{lemma}

\begin{proof}
Begin by noticing that the total variation of these measures is bounded, namely $|\eta_n|=\ell$ for all $n\in\Nat$.  Thus, there is a subsequence, also denoted by $(\eta_n)$, and $\eta\in\cM(\Real)$ such that $\eta_n\weakstar\eta$.  Since, for any measurable set $\cB$ we have that
\begin{align*}
\cL(\cB\cap\cC_{n+1})&=\cL(\cB\cap(\psi_1(\cC_{n})\cup\psi_2(\cC_{n})))=\cL(\cB\cap\psi_1(\cC_{n}))+\cL(\cB\cap\psi_2(\cC_{n}))\\
&=\tfrac 13 \cL(\psi_1^{-1}(\cB\cap\psi_1(\cC_{n})))+\tfrac 13 \cL(\psi_2^{-1}(\cB\cap\psi_2(\cC_{n})))\\
&=\tfrac 13 \cL(\psi_1^{-1}(\cB)\cap\cC_{n})+\tfrac 13 \cL(\psi_2^{-1}(\cB)\cap\cC_{n})
\end{align*}
and, hence,
\beqn
\tfrac{1}{2}(\psi_{1\#}\eta_n+\psi_{2\#}\eta_n)=\eta_{n+1}.
\eeqn
Thus, we can take the limit in $n$ to obtain
\beqn\label{measrel}
\tfrac{1}{2}(\psi_{1\#}\eta+\psi_{2\#}\eta)=\eta.
\eeqn
We now show that there is only one measure $\eta$ of total variation $\ell$ that satisfies this condition.


The value $s=\ln 2 / \ln 3$ is the Hausdorff dimension of the Cantor set and $0<\cH^s(\cC)<\infty$, see Falconer \cite{Falconer03}.  For any Borel set $\cB\subseteq [0,\ell]$ we have
\begin{align*}
\tfrac{1}{2}[(\psi_{1\#}\cH^s)(\cB)+(\psi_{2\#}\cH^s)(\cB)]&=\tfrac{1}{2}[\cH^s(\cC\cap\psi_1^{-1}(\cB))+\cH^s(\cC\cap\psi_2^{-1}(\cB))]\\
&=\tfrac{3^s}{2}[\cH^s(\psi_1(\cC)\cap \cB)+\cH^s(\psi_2(\cC)\cap \cB)]\\
&=\cH^s((\psi_1(\cC)\cup\psi_2(\cC)\cap\cB)\\
&=(\cH^s\mres\cC)(\cB).
\end{align*}
Thus, the measure $\eta=\frac{\ell}{\cH^s(\cC)}\cH^s\mres\cC$ satisfies \eqref{measrel}. 

To establish the uniqueness we will use the coincidence criterion, Proposition~1.8 in \cite{Ambrosio2000}.  We begin by showing that if $\eta$ satisfies \eqref{measrel} and $\eta(\Real)=\ell$, then $\eta(I_{n,i})=2^{-n}\ell$, which we prove by induction.  Clearly $\eta(I_{0,1})=\eta([0,\ell])=\ell$.  Now assume that $\eta(I_{n,i})=2^{-n}\ell$ for all $i\in\{1,\dots,2^n\}$ and let $I=I_{n+1,j}$ for some $j\in\{1,\dots,2^{n+1}\}$.  Notice that one of $\psi^{-1}_1(I)$ or $\psi^{-1}_2(I)$ is empty while the other is some $I_{n,i}$.  Thus, by the induction hypothesis and \eqref{measrel},
\begin{align}
\eta(I)=\tfrac{1}{2}[\psi_{1\#}\eta(I)+\psi_{2\#}\eta(I)]=\tfrac{1}{2}[\eta(\psi_1^{-1}(I))+\eta(\psi_1^{-1}(I))]=\tfrac{1}{2} 2^{-n}\ell=2^{-n-1}\ell.
\end{align}
It follows that there is only one positive measure of total variation $\ell$ that satisfies \eqref{measrel}.
\end{proof}

We now establish a compactness result that will be useful later.
\begin{theorem}\label{compactThm}
Consider a sequence $(u_n)\in\cW$ such that $\sup_{n\in\Nat} E_n(u_n)<\infty$.  It follows that there is a subsequence, also denoted by $(u_n)$, and a $u\in W^{1,1}(-\delta,\ell)$ with $u'\in BV(-\delta,\ell)$ such that
\begin{enumerate}
\item\label{c1} $u_n\rightarrow u\ \text{in}\ W^{1,1}(-\delta,\ell)$,
\item\label{c2} $u''_n\cL\weakstar u''\ \text{in}\ \cM(-\delta,\ell)$,
\item\label{c3} $u''\ll\eta$, where $\eta=\tfrac{\ell}{\cH^s(\cC)}\cH^s\mres\cC$ with $s=\ln 2 / \ln 3$.
\end{enumerate}
\end{theorem}

\begin{proof}
Since $u_n(0)=u_n'(0)=0$, it follows that
$$
|u_n(\ell)|\leq \int_0^\ell |u_n'(y)|dy\leq  \ell \int_{0}^\ell |u_n''(x)|dx,
$$
and applying Proposition~\ref{udpzero} and H\"older's inequality yields
\beqn
|u_n(\ell)|\leq\ell \int_{0}^\ell \chi_{\cC_n}|u_n''|d\cL\leq \ell \norm{\chi_{\cC_n}}_{L^2(\cC_n)} \norm{u_n''}_{L^2(\cC_n)}=\ell^{3/2}\big(\tfrac{2}{3}\big)^{n/2}\norm{u_n''}_{L^2(\cC_n)}.
\eeqn
By using this inequality and $\sup_n E_n(u_n)<+\infty$, for any $\alpha>0$ we find that
\beqn
\sup_{n\in\Nat}\int_{\cC_n}\big(\tfrac{2}{3}\big)^n\big( b-\alpha \ell^3\big)|u_n''|^2d\cL-\frac{P^2}{\alpha}\le\sup_{n\in\Nat}\int_0^\ell b_n|u_n''|^2d\cL-\alpha|u_n(\ell)|^2-\frac{P^2}{\alpha}<\infty,
\eeqn
where we also used Proposition \ref{udpzero}.
Choosing $\alpha=\frac{1}{2}\ell^{-3}b$ then shows
\beqn\label{L2bound}
\sup_{n\in\Nat} \big(\tfrac{2}{3}\big)^n\int_{\cC_n}|u_n''|^2d\cL<\infty.
\eeqn
From here we can apply H\"older's inequality to find that
\beqn\label{udpL1bound}
\sup_{n\in\Nat} \int_{-\delta}^\ell |u_n''|d\cL\leq\sup_{n\in\Nat} \norm{\chi_{\cC_n}}_{L^2(\cC_n)} \norm{u_n''}_{L^2(\cC_n)}=\sup_{n\in\Nat}\ \ell^{1/2}\big(\tfrac{2}{3}\big)^{n/2}\norm{u_n''}_{L^2(\cC_n)}<\infty.
\eeqn
Since $u_n(0)=u_n'(0)=0$, this also implies
\beqn\label{uW11bound}
\sup_{n\in\Nat}\int_{-\delta}^\ell |u_n|d\cL<\infty\qquad\text{and}\qquad  \sup_{n\in\Nat}\int_{-\delta}^\ell |u_n'|d\cL<\infty.
\eeqn
It follows from \eqref{udpL1bound} and \eqref{uW11bound} that there is a $u\in W^{1,1}(-\delta,\ell)$ such that, up to a subsequence, $u_n\to u$ in $W^{1,1}(-\delta,\ell)$.  Moreover, from \eqref{udpL1bound} and \eqref{uW11bound}$_2$ that there is a $v\in BV(-\delta,\ell)$ such that, up to subsequence, $u'_n\weakstar v\in BV(-\delta,\ell)$.  Putting these two results together we obtain $v = u'$.  

To establish Item \ref{c3}, begin by considering the sequence of measures defined by $\eta_n\coloneqq\big(\tfrac{3}{2}\big)^n \cL\mres \cC_n$.  From Lemma~\ref{etaconv} we know that $\eta_n\weakstar \eta$, where \eqref{measrep} holds.  Now notice that for $n\in\Nat$ and a Borel set $\cB\subseteq (-\delta,\ell)$,
\beqn
(u''_n\cL)(\cB)=\int_\cB u''_nd\cL=\int_{\cB\cap \cC_n} u''_nd\cL
= \int_\cB\big(\tfrac{2}{3}\big)^n u''_n d\eta_n,
\eeqn
from which it follows that $u''_n\cL\ll\eta_n$, and so we can define $\gamma_n\coloneqq\frac{u''_n\cL}{\eta_n}=(\tfrac{2}{3})^n u_n''$.  Moreover, we have
\beqn\label{Eusingeta}
\int_{-\delta}^\ell \gamma_n^2 d\eta_n=\int_{\cC_n} \big(\tfrac{2}{3}\big)^{2n}|u_n''|^2 (\tfrac{3}{2})^{n}d\cL= \int_{\cC_n}\big(\tfrac{2}{3}\big)^{n}|u_n''|^2d\cL.
\eeqn
Thus, it follows from \eqref{L2bound} that
\beqn
\sup_{n\in\Nat} \int_{-\delta}^\ell \gamma_n^2 d\eta_n<\infty.
\eeqn
Using Lemma~\ref{abscontm} with $\mu_n\coloneqq u''_n\cL$ and $\mu=u''$, we now find that $u''\ll\eta$.
\end{proof}

Motivated by the previous result, we consider the following function space:
\beqn
\cW_\infty\coloneqq\{u\in W^{1,1}(-\delta,\ell)\ |\ u''\in\cM(-\delta,\ell),\ u''\ll\cH^s\mres\cC,\ \text{and}\ u=0\ \text{on}\ (-\delta,0]\}.
\eeqn
We can now state the $\Gamma$-convergence result using the topology determined by Items \ref{c1} and \ref{c2} of the previous result. For the theory of $\Gamma$-convergence and its properties we refer to Dal Maso \cite{dalmaso}.
\begin{theorem}
Let $s=\ln 2 / \ln 3$ and
\beqn\label{Endef}
E_\infty(u)\coloneqq
\begin{cases}
\displaystyle\int_\cC \frac{b\cH^s(\cC)}{2\ell} \Big(\frac{u''}{\cH^s\mres\cC}\Big)^2 d\cH^s-Pu(\ell) & \mbox{if } u\in \cW_\infty,\\
+\infty & \mbox{if } u\in  W^{2,2}(-\delta,\ell)\setminus \cW_\infty.
\end{cases}
\eeqn
The sequence of energies $E_n$ defined in \eqref{Endef} $\Gamma$-converges to $E_\infty$ in the following sense:

\begin{itemize}
\item[i)] (liminf inequality) for every $u\in W^{2,2}(-\delta,\ell)$ and every sequence $(u_n)\in W^{2,2}(-\delta,\ell)$ such that $u_n\to u$ in $W^{1,1}(-\delta,\ell)$ and  $u''_n\cL\weakstar u''$ in $\cM(-\delta,\ell)$, we have that 
$$
\liminf_{n\to+\infty}E_n(u_n)\ge E_\infty(u);
$$
\item[ii)] (recovery sequence) for every $u\in W^{2,2}(-\delta,\ell)$ there exists a sequence $(u_n)\in W^{2,2}(-\delta,\ell)$ such that $u_n\to u$ in $W^{1,1}(-\delta,\ell)$ and  $u''_n\cL\weakstar u''$ in $\cM(-\delta,\ell)$, and
\beqn\label{limsupeq}
\limsup_{n\to+\infty}E_n(u_n)\le E_\infty(u).
\eeqn
\end{itemize}
\end{theorem}

\begin{proof}
For notational convenience, we set $\eta_n\coloneqq\big(\tfrac{3}{2}\big)^n \cL\mres \cC_n$, $n\in\Nat$, and $\eta\coloneqq\frac{\ell}{\cH^s(\cC)}\cH^s\mres\cC$.

First we establish the liminf inequality.  Let $(u_n)\in W^{2,2}(-\delta,\ell)$  and $u\in W^{2,2}(-\delta,\ell)$ such that $u_n\to u$ in $W^{1,1}(-\delta,\ell)$ and  $u''_n\cL\weakstar u''$ in $\cM(-\delta,\ell)$.  
We can assume that $\liminf_{n\rightarrow \infty}E_n(u_n)$ is finite, otherwise there is nothing to prove.  In this case, by possibly passing to a subsequence, we can assume that $\sup_{n\in\Nat} E_n(u_n)<\infty$. Thus $(u_n)\in \cW$ and from Theorem~\ref{compactThm} it follows that $u\in\cW_\infty$; in particular, $u''\ll\eta$.  Moreover, from Lemma~\ref{abscontm} with $\mu_n:=u_n'' \cL$ and $\mu=u''$, we have
\beqn
\liminf_{n\rightarrow\infty} \int_{-\delta}^\ell \Big(\frac{u''_n\cL}{\eta_n}\Big)^2d\eta_n\geq \int_{-\delta}^\ell \Big(\frac{u''}{\eta}\Big)^2d\eta.
\eeqn
Utilizing this and \eqref{Eusingeta} we easily obtain
\begin{multline}
\liminf_{n\rightarrow\infty} E_n(u_n)=\liminf_{n\rightarrow\infty} \int_{\delta}^\ell \frac{b}{2}  \Big(\frac{u''_n\cL}{\eta_n}\Big)^2 d\eta_n-Pu(\ell)\\
\geq \int_{-\delta}^\ell \frac{b}{2} \Big(\frac{u''}{\eta}\Big)^2d\eta-Pu(\ell)=\int_\cC \frac{b\cH^s(\cC)}{2\ell} \Big(\frac{u''}{\cH^s\mres\cC}\Big)^2 d\cH^s-Pu(\ell).
\end{multline}

To establish the recovery sequence, it suffices to consider $u\in \cW_\infty$ for which $E_\infty(u)<+\infty$.  Set $\gamma\coloneqq\frac{u''}{\eta}$.  Recalling the notation in \eqref{Cnint}, define $u_n\in\cW$ so that
\beqn\label{undef}
u_n''(x):=\sum_{i=1}^{2^n}\dashint_{I_{n,i}\cap\cC} \big(\tfrac{3}{2}\big)^n \gamma d\cH^s \chi_{I_{n,i}}(x),\quad x\in(0,\ell).
\eeqn
Notice that $u_n''=0$ off of $\cC_n$.  First we show that $u_n''\cL\weakstar u''$.  Towards this end, let $\varphi\in C_c(-\delta,\ell)$ and using \eqref{Cnint}$_1$ we have
\begin{align}
\int_{-\delta}^\ell \varphi u_n''d\cL&=\sum_{j=1}^{2^n}\int_{I_{n,j}} \varphi u_n''d\cL\\
&= \sum_{j=1}^{2^n} \Big(\int_{I_{n,j}} \varphi d\cL\, \cH^s(I_{n,j}\cap\cC)^{-1}\big(\tfrac{3}{2}\big)^n\int_{I_{n,j}\cap\cC} \gamma d\cH^s\Big).
\end{align}
From \eqref{measrep} we deduce that 
\beqn\label{ICmeas}
\cH^s(I_{n,i}\cap\cC)=\cH^s(\cC)2^{-n},
\eeqn
thus, upon setting
\beqn
\varphi_n\coloneqq\sum_{i=1}^{2^n} \dashint_{I_{n,i}} \varphi d\cL \chi_{I_{n,i}},
\eeqn
we obtain that
\beqn
\int_{-\delta}^\ell \varphi u_n''d\cL=
 \sum_{j=1}^{2^n} \Big(\dashint_{I_{n,j}} \varphi d\cL\, \frac{\ell}{\cH^s(\cC)}\int_{I_{n,j}\cap\cC} \gamma d\cH^s\Big)
=\int_\cC \varphi_n \gamma d\eta=\int_{-\delta}^\ell \varphi_n du''.
\eeqn
Since $E_\infty(u)<+\infty$ we have that $\frac{u''}{\cH^s\mres\cC}$ is a square integrable function on $(-\delta,\ell)$ with respect to the measure $\cH^s\mres\cC$. Thus, to conclude our proof that $u_n''\cL\weakstar u''$ it suffices to show that $\varphi_n\rightarrow\varphi$ in $L^2(-\delta,\ell)$ with respect to the measure $\cH^s\mres\cC$.  To establish this, begin by fixing $\ve>0$ and find $\delta>0$ such that if $|x-y|<\delta$, then $|\varphi(x)-\varphi(y)|<\ve$.  Then, for $n$ large enough so that $3^{-n}\ell <\delta$, we have
\begin{align}
\int_\cC |\varphi_n-\varphi|^2 d\cH^s = \sum_{i=1}^{2^n} \int_{I_{n,i}\cap\cC} \Big| \dashint_{I_{n,i}}\varphi d\cL-\varphi\Big|^2 d\cH^s\leq \sum_{i=1}^{2^n} \cH^s(I_{n,i}\cap\cC)\ve^2=\cH^s(\cC)\ve^2,
\end{align}
which establishes the desired convergence.  

Next we show that $u'_n\rightarrow u'$ in $L^1(-\delta,\ell)$.  We have $u'_n(x)=0$ for $x\in(-\delta,0]$ and
\beqn
u'_n(x)=\int_0^x u_n'' d\cL.
\eeqn
Since $u_n''\cL\weakstar u''$ and $u''$ has no atoms, it follows that $u_n'\rightarrow u'$ pointwise.  Moreover, since
\beqn
|u'_n(x)-u'(x)|\leq |u''_n\cL([0,x])|+|u''([0,x])|\leq 2 |u''|([0,\ell]),
\eeqn
the dominated convergence theorem implies that $u_n'\rightarrow u'$ in $L^1(-\delta,\ell)$.  Since $u_n(0)=u(0)=0$, it immediately follows that $u_n\rightarrow u$ in $L^1(-\delta,\ell)$.  Thus, the sequence $(u_n)$ converges to $u$ in the desired topology.

We can now establish the recovery sequence condition.  A calculation using \eqref{undef}, \eqref{ICmeas}, Jensen's inequality, and the definition $\gamma=\frac{u''}{\eta}$ yields
\begin{align*}
\int_{\cC_n} \big(\tfrac{2}{3}\big)^n|u_n''|^2 d\cL&=\sum_{i=1}^{2^n} \int_{I_{n,i}} \big(\tfrac{2}{3}\big)^n \Big[ \dashint_{I_{n,i}\cap \cC}\big(\tfrac{3}{2}\big)^n\gamma d\cH^s \Big]^2 d\cL\\
&\leq \sum_{i=1}^{2^n}\big(\tfrac{3}{2}\big)^n |I_{n,i}| \cH^s(I_{n,i}\cap\cC)^{-1} \int_{I_{n,i}\cap\cC} \gamma^2 d\cH^s\\
&= \sum_{i=1}^{2^n}\big(\tfrac{3}{2}\big)^n \tfrac{\ell}{3^n} \tfrac{2^n}{\cH^s(\cC)} \int_{I_{n,i}\cap\cC} \gamma^2 d\cH^s\\
&\le\frac{\ell}{\cH^s(\cC)}\int_\cC\gamma^2 d\cH^s
=\frac{\cH^s(\cC)}{\ell} \int_\cC \Big(\frac{u''}{\cH^s\mres\cC}\Big)^2 d\cH^s.
\end{align*}
It now follows from the definition of the energy $E_n$ in \eqref{Endef} that \eqref{limsupeq} holds.
\end{proof}

The stationary points of the limiting energy $E_\infty$ can be characterized exactly, as the next result shows.
\begin{proposition}
If $u\in\cW_\infty$ is a stationary point of $E_\infty$, then
\beqn
u(x)=\frac{\ell P}{b\cH^s(\cC)} \int_0^x \int_{\cC\cap[0,y]} (\ell-z) d\cH^s(z)dy,\quad x\in(0,\ell).
\eeqn
\end{proposition}

\begin{proof}
We set $\xi\coloneqq\frac{u''}{\cH^s\mres\cC}$. Begin by noticing that if $u\in\cW_\infty$ than
\begin{align}
u(\ell)&=\int_0^\ell \int_{[0,x]\cap\cC} \xi(z)\,d\cH^s(z)dx=\int_\cC\int_0^\ell \chi_{[0,x]}(z)\xi(z)dx\,d\cH^s(z)=\int_\cC (\ell-z)\xi(z)\,d\cH^s(z).
\end{align}
Substituting this into the formula for $E_\infty$ shows that
\beqn
E_\infty(u)=\int_\cC\Big[ \frac{b\cH^s(\cC)}{2\ell} \xi(z)^2-P(\ell-z)\xi(z)\Big]d\cH^s(z).
\eeqn
A standard argument implies that the Euler--Lagrange equation associated with $E_\infty$ is
\beqn
\frac{b\cH^s(\cC)}{\ell}\xi(z)-P(\ell-z)=0,\qquad z\in(0,\ell)
\eeqn
and, so, for $x\in (0,\ell)$
\beqn
u(x)=\int_0^x\int_{\cC\cap[0,y]} \xi(z) d\cH^s(z) dy=\frac{\ell P}{b\cH^s(\cC)} \int_0^x \int_{\cC\cap[0,y]} (\ell-z) d\cH^s(z)dy.
\eeqn
\end{proof}

\section*{Acknowledgments}
RP acknowledge the support from the University of Pisa through the project PRA2022-69 and that of
the Italian National Group of Mathematical Physics of INdAM.

\bibliographystyle{acm}
\bibliography{smallbendbib}

\end{document}